\documentclass[12pt]{amsart}
\usepackage{amscd,amsmath,amsthm,amssymb,graphics}
\usepackage{pstcol,pst-plot,pst-3d}%\usepackage[T1]{fontenc}
\usepackage{lmodern,pst-node}%\usepackage{geometric}
\usepackage{multicol}
\usepackage{amssymb}
\usepackage{graphicx}
\usepackage{tikz}
\usepackage{subfig}
\usepackage{mathtools}
\usepackage{tikz-3dplot}
\usepackage{enumitem}
\usepackage{graphicx}
\usepackage{tikz}
\usepackage{enumitem}
\usepackage{enumerate}
\usepackage{epic,eepic}
\usepackage{amsfonts,amssymb,amscd,amsmath,enumerate,verbatim}
\psset{unit=0.7cm,linewidth=0.8pt,arrowsize=2.5pt 4}
\newpsstyle{fatline}{linewidth=1.5pt}
\newpsstyle{fyp}{fillstyle=solid,fillcolor=verylight}
\definecolor{verylight}{gray}{0.97}
\definecolor{light}{gray}{0.9}
\definecolor{medium}{gray}{0.85}
\definecolor{dark}{gray}{0.6}

\def\frk{\frak}               % font for "Fraktur"

\def\Phi{{\frk n}}
\def\Phi{{\frk N}}
\def\MI{{\mathcal I}}

\def\MG{{\mathcal G}}

\def\opn#1#2{\def#1{\operatorname{#2}}} % to make operators
\opn\chara{char} \opn\length{\ell} \opn\pd{pd} \opn\rk{rk}
\opn\projdim{proj\,dim} \opn\injdim{inj\,dim} \opn\rank{rank}
\opn\depth{depth} \opn\grade{grade} \opn\height{height}
\opn\embdim{emb\,dim} \opn\codim{codim}

\opn\Tr{Tr} \opn\bigrank{big\,rank}
\opn\superheight{superheight}\opn\lcm{lcm}
\opn\trdeg{tr\,deg}%\emph{
\opn\reg{reg} \opn\lreg{lreg} \opn\ini{in} \opn\lpd{lpd}
\opn\size{size}\opn\bigsize{bigsize}
\opn\cosize{cosize}\opn\bigcosize{bigcosize}
\opn\sdepth{sdepth}\opn\sreg{sreg}
\opn\link{link}\opn\fdepth{fdepth}
\opn\div{div} \opn\Div{Div} \opn\cl{cl} \opn\Cl{Cl}
\opn\Spec{Spec} \opn\Supp{Supp} \opn\supp{supp} \opn\Sing{Sing}
\opn\Ass{Ass} \opn\Min{Min}\opn\Mon{Mon} \opn\dstab{dstab} \opn\astab{astab}
\opn\Syz{Syz}
\opn\Ann{Ann} \opn\Rad{Rad} \opn\Soc{Soc}
\opn\Im{Im} \opn\Ker{Ker} \opn\Coker{Coker} \opn\Am{Am}
\opn\Hom{Hom} \opn\Tor{Tor} \opn\Ext{Ext} \opn\End{End}
\opn\Aut{Aut} \opn\id{id}

\opn\nat{nat}
\opn\pff{pf}%   \pf exists already
\opn\Pf{Pf} \opn\GL{GL} \opn\SL{SL} \opn\mod{mod} \opn\ord{ord}
\opn\Gin{Gin} \opn\Hilb{Hilb}\opn\sort{sort}
\opn\S{S}
\opn\aff{aff} \opn\con{conv} \opn\relint{relint} \opn\st{st}
\opn\lk{lk} \opn\cn{cn} \opn\core{core} \opn\vol{vol}
\opn\link{link} \opn\star{star}\opn\lex{lex}\opn\ini{in}
\opn\gr{gr}

\def\pot#1#2{#1[\kern-0.28ex[#2]\kern-0.28ex]}

\opn\dirlim{\underrightarrow{\lim}}
\opn\inivlim{\underleftarrow{\lim}}
\let\to=\rightarrow

\def\Implies{\ifmmode\Longrightarrow \else
        \unskip${}\Longrightarrow{}$\ignorespaces\fi}
\def\implies{\ifmmode\Rightarrow \else
        \unskip${}\Rightarrow{}$\ignorespaces\fi}
\def\iff{\ifmmode\Longleftrightarrow \else
        \unskip${}\Longleftrightarrow{}$\ignorespaces\fi}

\let\:=\colon
\newtheorem{Theorem}{Theorem}[section]
 
 \newtheorem{Corollary}[Theorem]{Corollary}
 \newtheorem{Proposition}[Theorem]{Proposition}
 \newtheorem{Remark}[Theorem]{Remark}
 
 \newtheorem{Example}[Theorem]{Example}

\let\epsilon\varepsilon
\let\kappa=\varkappa
\def\qed{\ifhmode\textqed\fi
      \ifmmode\ifinner\quad\qedsymbol\else\dispqed\fi\fi}
\def\textqed{\unskip\nobreak\penalty50
       \hskip2em\hbox{}\nobreak\hfil\qedsymbol
       \parfillskip=0pt \finalhyphendemerits=0}
\def\dispqed{\rlap{\qquad\qedsymbol}}

\opn\dis{dis}
\def\pnt{{\raise0.5mm\hbox{\large\bf.}}}

\opn\Lex{Lex}

\begin{document}
 \title {On the regularity of join-meet ideals of modular lattices}

 \author {Rodica Dinu, Viviana Ene, Takayuki Hibi}

\address{Rodica Dinu, Faculty of Mathematics and Computer Science, University of Bucharest, Str. Academiei 14, Bucharest -- 010014, Romania}
\email{rodica.dinu93@yahoo.com}

\address{Viviana Ene, Faculty of Mathematics and Computer Science, Ovidius University, Bd.\ Mamaia 124,
 900527 Constanta, Romania}\email{vivian@univ-ovidius.ro}

\address{Takayuki Hibi, Department of Pure and Applied Mathematics, Graduate School of Information Science and Technology,
Osaka University, Suita, Osaka 565-0871, Japan}
\email{hibi@math.sci.osaka-u.ac.jp}

%\thanks{The first author was supported by the grant UEFISCDI,  PN-II-ID-PCE- 2011-3-1023.}

 \begin{abstract} We study join-meet ideals associated with modular non-distributive lattices. We give a lower bound for the regularity and show that they are not linearly related.
 \end{abstract}

\thanks{}

 \maketitle

\section*{Introduction}

Let $L$ be a finite lattice and $K[L]=K[x_a: a\in L]$ be the polynomial ring over the field $K.$ The ideal 
\[
I_L=(f_{ab}:=x_ax_b-x_{a\wedge b}x_{a\vee b}: a,b\in L, a,b \text{ incomparable})\subset K[L]
\]
is called the \textit{join-meet ideal} associated with $L.$ 

It was proved in \cite{Hibi} that $I_L$ is a prime ideal if and only if $L$ is a distributive lattice. In this case, the ring 
$R[L]=K[L]/I_L$ is a normal Cohen-Macaulay domain; see \cite{Hibi}. In  literature, the ring $R[L]$ is known as the Hibi ring 
of the distributive lattice $L$. Hibi rings and the corresponding varieties appear naturally in different combinatorial,  algebraic, and geometric contexts; see \cite{EHH, EHHM15, Hibi, howe,  LB, LM, Wa}. Paper \cite{kim2} surveys the relation between Hibi rings and representation theory.

The starting point of this work was the question whether there exist non-distribu\-tive lattices $L$ such that $I_L$ has a linear resolution. If $L$ is distributive, by  Birkhoff's theorem \cite{B}, it follows that $L$ is the lattice of poset ideals $\MI(P)$ of the poset $P$ consisting of the join-irreducible elements of $L.$ It was proved in \cite{ERQ} that $I_L$ has a linear resolution if and only if 
$L=\MI(P)$ where $P$ is the sum of a chain with an isolated element. This result has been recovered in a later paper \cite{EHM} where it was proved a much stronger result, namely,  $\reg I_L=|P|-\rank P.$ Much less is known about join-meet ideals of non-distributive lattices. As it was observed in \cite{EnHi}, if $L$ is modular, but not distributive, the associated join-meet ideal may be not radical.  

In this paper we consider join-meet ideals of finite modular lattices. We recall that a finite lattice $L$ is  modular if it does not 
contain any sublattice isomorphic to the pentagon lattice of Figure~\ref{pentagon}. For basic properties of lattices, like distributivity and modularity, we refer the reader to the  monographs
\cite{B} and \cite{Sta}.

\begin{figure}[bht]
\begin{center}
\psset{unit=0.5cm}
\begin{pspicture}(-15.3,-2.5)(-1,3.5)

\psline(-9,3)(-11,1)
%\psline(-9,3)(-9,1)
\psline(-9,3)(-7,2)
\psline(-7,2)(-7,0)
\psline(-11,1)(-9,-1)
%\psline(-9,1)(-9,-1)
\psline(-7,0)(-9,-1)

\rput(-9,3){$\bullet$}
\put(-8.9,3.3){$e$}
\rput(-11,1){$\bullet$}
\put(-11.6,0.9){$b$}
%\rput(-9,1){$\bullet$}
%\put(-8.6,0.9){$c$}
\rput(-7,2){$\bullet$}
\put(-6.7,2){$d$}
\rput(-9,-1){$\bullet$}
\put(-9.1,-1.6){$a$}
\rput(-7,0){$\bullet$}
\put(-6.7,0){$c$}
\end{pspicture}
\end{center}
\caption{Pentagon lattice}\label{pentagon}
\end{figure}
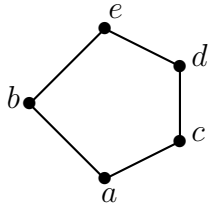

Every modular lattice $L$ possesses a rank function with the property that 
\[\rank a+ \rank b=\rank a\wedge b+ \rank a\vee b, \text{ for all  }a,b\in L.\]

If $L$ is a modular lattice, then $L$ is not distributive if and only if it contains a sublattice isomorphic to the diamond 
lattice of Figure~\ref{diamond}. 
\begin{figure}[bht]
\begin{center}
\psset{unit=0.5cm}
\begin{pspicture}(-15.3,-2.5)(-1,3.5)

\psline(-9,3)(-11,1)
\psline(-9,3)(-9,1)
\psline(-9,3)(-7,1)
\psline(-11,1)(-9,-1)
\psline(-9,1)(-9,-1)
\psline(-7,1)(-9,-1)

\rput(-9,3){$\bullet$}
\put(-9.1,3.3){$e$}
\rput(-11,1){$\bullet$}
\put(-11.6,0.9){$b$}
\rput(-9,1){$\bullet$}
\put(-8.6,0.9){$c$}
\rput(-7,1){$\bullet$}
\put(-6.7,1){$d$}
\rput(-9,-1){$\bullet$}
\put(-9.1,-1.6){$a$}

\end{pspicture}
\end{center}
\caption{Diamond lattice}\label{diamond}
\end{figure}
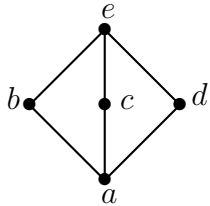

By \cite[Lemma 1.2]{EnHi}, a modular lattice contains a diamond sublattice $L^\prime$ such that 
$\rank \max L^\prime - \rank \min L^\prime =2.$ This result leads to considering diamond-type lattices with an arbitrary number 
of elements, say $y_1,\ldots,y_n$ in between the maximal and the minimal elements. We denote such a lattice by $D_{n+2}.$  In Section~\ref{one} we study the join-meet ideal of $D_{n+2.}$ We show that this ideal is Gorenstein (Theorem~\ref{Gorenstein}). Moreover, we   compute its regularity (Proposition~\ref{CM}) and show
that it is not linearly related (Proposition~\ref{notlin}). An equigenerated graded ideal $I$ in a polynomial ring $S$ is called \textit{linearly related} if  it has linear relations. 

In Section~\ref{two}, we show that if $L$ is a finite lattice and $L^\prime$ is an \textit{induced sublattice} of $L$ (see Section~\ref{two} for the definition), 
then $K[L^\prime]/I_{L^\prime}$ is an algebra retract of $K[L]/I_L.$ This leads to the main result of this paper, Theorem~\ref{mainthm}, which states that  if $L$ is a modular non-distributive lattice, then $\reg(K[L]/I_L)\geq 3$ and, moreover, $I_L$ is not linearly related. 
In particular, the join-meet ideal of a modular non-distributive lattice does not have a linear resolution.

Therefore, combining Theorem~\ref{mainthm} with \cite[Corollary 10]{ERQ} or \cite[Corollary 1.2]{EHM}, it follows that if $L$ is a finite modular lattice, then $I_L$ has a linear resolution if and only if $L$ is distributive 
and the poset of join-irreducible elements of $L$ is the sum of a chain with an isolated element.

In addition, Theorem~\ref{mainthm} shows that, for a modular lattice $L$, if $I_L$ is linearly related, then $L$ must be distributive. The planar distributive lattices with linearly related join-meet ideals are completely classified in \cite[Theorem 4.12]{En}. On the other hand, there are non-planar distributive lattices whose join-meet ideals are linearly related. For example, if $L=B_n$ is the Boolean 
lattice of rank $n,$ it is easily seen that the associated Hibi ring $R[L]=K[L]/I_L$ is in fact the Segre product of $n$ polynomials rings in two variables. In \cite[Theorem 6]{Ru} it was shown that the defining ideal of this Segre product is linearly related. Consequently, 
$I_L$ is linearly related for $L=B_n,$ $n\geq 2.$

However, classifying all the distributive lattices whose join-meet ideal is linearly related remains open.

\section*{Acknowledgement}
We thank Mateusz Micha\l ek for drawing our attention to the paper \cite{Ru} and Florin Ambro for valuable discussions. 
We  gratefully acknowledge the use of the Singular  software \cite{GPS}  for  computer experiments.

\section{Diamond lattice $D_{n+2}$}
\label{one}

 Let us consider the diamond lattice $D_{n+2}, n\geq 3,$ with  elements $x> y_1,\ldots,y_n>z.$ In Figure \ref{D7}, we 
displayed the diamond lattice $D_7.$

\begin{figure}[hbt]
\begin{center}
\psset{unit=1cm}
\begin{pspicture}(-2,0)(3,2)
\rput(0,0){$\bullet$}
\rput(0,1){$\bullet$}
\rput(0,2){$\bullet$}
\rput(1,1){$\bullet$}
\rput(2,1){$\bullet$}
\rput(-1,1){$\bullet$}
\rput(-2,1){$\bullet$}
\psline(0,0)(0,1)
\psline(0,1)(0,2)
\psline(0,0)(1,1)
\psline(0,0)(2,1)
\psline(0,0)(-1,1)
\psline(0,0)(-2,1)
\psline(0,2)(0,1)
\psline(0,2)(1,1)
\psline(0,2)(2,1)
\psline(0,2)(-1,1)
\psline(0,2)(-2,1)
\end{pspicture}
\end{center}
\caption{Diamond lattice $D_7$}\label{D7}
\end{figure}
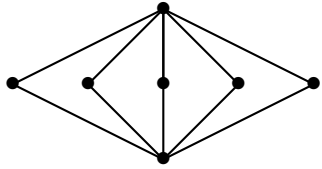

Let $K$ be a field.
The join-meet ideal associated to the diamond lattice is
$$
I_{D_{n+2}}= (y_iy_j-xz: 1\leq i<j \leq n)\subseteq K[x, y_1,\ldots, y_n, z].
$$
We denote it simply by $I$ and, also, we denote $S=K[x, y_1,\ldots, y_n, z]$ the ring of polynomials in $n+2$ indeterminates over $K.$ 
We consider the following order of the variables $x>y_1>\cdots>y_n>z$.

We set $f_{ij}:=y_iy_j-xz$  for $1\leq i<j\leq n$, $g_i:=xz(y_i-y_n)$ for $1\leq i\leq n-1,$ and $h:=xz(y_n^2-xz)$.

\begin{Proposition} The reduced Gr\"obner basis of $I$ with respect to the reverse lexicographic order induced by $x>y_1>y_2>\cdots >y_n>z$ is
$$
\mathcal{G}=\{f_{ij}, g_k, h: 1\leq i<j\leq n, 1\leq k\leq n-1\}.
$$
\end{Proposition}

\begin{proof}
We apply the Buchberger's Algorithm. Obviously, $\MG$ generates $I$. Straightforward calculation shows that all the $S$-polynomials reduce to $0$ with respect to $\MG,$ thus $\MG$ is a Gr\"obner basis of $I.$ Moreover, $\MG$ is obviously reduced.

\end{proof}

\begin{Corollary} The initial ideal of $I$ with respect to the reverse lexicographic order is 
 \[\ini_{<}(I)=(y_iy_j, xy_kz, xy_n^2z: 1\leq i<j\leq n, 1\leq k\leq n-1).\]
\end{Corollary}

\begin{Proposition}\label{Hilbert} The Hilbert series of $S/I$ and $S/\ini_{<}(I)$ is
$$
\Hilb_{S/I}(t)=\Hilb_{S/\ini_<(I)}=\frac{1+nt+nt^2+t^3}{(1-t)^2}.
$$
\end{Proposition}

\begin{proof} It is well known that $S/I$ and $S/\ini_<(I)$ have the same Hilbert series; see, for example, \cite[Corollary 6.1.5]{HH10}. 
In order to compute the Hilbert series of $S/\ini_<(I)$, we consider the following exact sequence:

\begin{equation}\label{sir0}
0\rightarrow \frac{S}{(\ini_{<}(I):xz)}(-2) \rightarrow \frac{S}{\ini_{<}(I)} \rightarrow \frac{S}{(\ini_{<}(I),xz)} \rightarrow 0,
\end{equation}
where the first non-zero map is the multiplication by $xz.$
Then \[\Hilb_{\frac{S}{\ini_{<}(I)}}(t)=t^2 \Hilb_{\frac{S}{(\ini_{<}(I):xz)}}(t)+\Hilb_{\frac{S}{(\ini_{<}(I),xz)}}(t).\]

 Since $S/(\ini_{<}(I):xz)= S/(y_1, y_2,\ldots,y_{n-1}, y_n^2)$ and $y_1, y_2,\ldots,y_{n-1}, y_n^2$ is a regular sequence in $S$, 
it follows that \[\Hilb_{\frac{S}{(\ini_{<}(I):xz)}}(t)=\frac{(1-t)^{n-1}(1-t^2)}{(1-t)^{n+2}}=\frac{1+t}{(1-t)^2}.\]

On the other hand, 
\begin{equation}\label{eq2}
\frac{S}{(\ini_<(I),xz)}=\frac{K[y_1,\ldots,y_n]}{(y_iy_j: 1\leq i<j\leq n)}\otimes_K\frac{K[x,z]}{(xz)},
\end{equation}
which implies that \[\Hilb_{\frac{S}{(\ini_{<}(I),xz)}}(t)=\left(1+n\frac{t}{1-t}\right)\frac{1+t}{1-t}=\frac{1+nt+(n-1)t^2}{(1-t)^2.}\]

Therefore, \[\Hilb_{\frac{S}{\ini_{<}(I)}}(t)=t^2\frac{1+t}{(1-t)^2}+\frac{1+nt+(n-1)t^2}{(1-t)^2}=\frac{1+nt+nt^2+t^3}{(1-t)^2}.\]
\end{proof}

\begin{Proposition}\label{CM}
\begin{itemize}
\item [{\em (i)}] The ideals $I$ and $\ini_{<}(I)$ are Cohen-Macaulay of dimension $2$.
\item [{\em (ii)}] We have $\reg(S/I)=\reg(S/\ini_<(I))=3.$
\item [{\em (iii)}] We have $\beta_{n,n+3}(S/I)=\beta_{n,n+3}(S/\ini_<(I))=1.$
\end{itemize}
\end{Proposition}

\begin{proof}
(i). The formula of the Hilbert series of $S/I$ from Proposition~\ref{Hilbert} shows that $\dim (S/I)=\dim (S/\ini_<(I))=2. $
By \cite[Corollary 3.3.5]{HH10}, it is enough to show that $S/\ini_<(I)$ is Cohen-Macaulay. But this follows easily from the exact sequence 
(\ref{sir0}) since the ends are Cohen-Macaulay of dimension $2$, therefore the middle term is also Cohen-Macaulay. Indeed, the left end in (\ref{sir0}) is even a complete intersection, while the right end is the tensor product from (\ref{eq2}) where the second factor is obviously Cohen-Macaulay and the first factor is  the Stanley-Reisner ring of the simplicial complex whose facets are $\{y_1\},\ldots, \{y_n\}$ which is Cohen-Macaulay.

(ii). Standard arguments for Cohen-Macaulay algebras show that  $\reg (S/I)=\reg(S/\ini_<(I))=3$ since the degree of the numerator of the Hilbert series of $S/I$ is equal to $3$; for more explanation on how to compute the regularity of a standard graded algebra  see, for example, \cite[Subsection 3.1]{En}.

(iii). The highest shift in the resolutions of $S/I$ and $S/\ini_<(I)$ appears in homological degree $n$ and $\beta_{n,n+3}(S/I)=\beta_{n,n+3}(S/\ini_<(I))=1$ since the leading coefficient  in the numerator of the Hilbert series is equal to $1$; see the discussion in \cite[Section 1]{EHHM15}.
\end{proof}

\begin{Remark}\label{Rem1}{\em 
One may easily observe that $\ini_<(I)$ has linear quotients if we order its minimal generators as follows:
\[y_1y_2,y_1y_3,\ldots, y_1y_n,y_2y_3,\ldots, y_2y_n,\ldots, y_{n-1}y_n,xy_1z,xy_2z,\ldots,xy_{n-1}z,xy_n^2z.\]

By using \cite[Exercise 8.8]{HH10}, one may compute all the graded Betti numbers of $S/\ini_<(I)$. 

In what follows, we will exploit only that $\beta_{n,n+1}(S/\ini_<(I))=0.$ 
Indeed, by \cite[Exercise 8.8]{HH10}, we have 
\[\beta_{i+1,i+j}(S/\ini_<(I))=\sum_{k, \deg f_k=j}\binom{r_k}{i}\] where $f_1,\ldots,f_m$ are the generators of $\ini_<(I)$ with
 $m=n(n+1)/2$ and, for each $k,$ $r_k$ denotes the number of variables which generate the ideal quotient $(f_1,\ldots,f_{k-1}):f_k.$ 
Then, for $i=n-1$ and $j=2,$ we get 
\[\beta_{n,n+1}(S/\ini_<(I))=\sum_{k, \deg f_k=2}\binom{r_k}{n-1}.\] One easily sees that, for every $k$ with  $\deg f_k=2,$  the ideal quotient 
$(f_1,\ldots,f_{k-1}):f_k$ is generated by at most $n-2$ variables, therefore $\beta_{n,n+1}(S/\ini_<(I))=0.$

Note that, by using similar arguments, we get 
\[\beta_{n,n+2}(S/\ini_<(I))=\sum_{k, \deg f_k=3}\binom{r_k}{n-1}=\sum_{k, \deg f_k=3}\binom{n-1}{n-1}=n-1.\]
}
\end{Remark}

\begin{Theorem}\label{Gorenstein}
The ideal $I$ is  Gorenstein.
\end{Theorem}

\begin{proof}
In order to prove that $\frac{S}{I}$ is a Gorenstein ring, we only need to show that $\beta_{n,n+1}(S/I)=\beta_{n,n+2}(S/I)=0,$ since, by 
Proposition~\ref{CM}, we already know that $\beta_{n,n+3}(S/I)=1.$ By Remark~\ref{Rem1}, we have $\beta_{n,n+1}(S/\ini_<(I))=0.$ Since
$\beta_{n,n+1}(S/I) \leq \beta_{n,n+1}(S/\ini_<(I)),$ we get the desired conclusion for $\beta_{n,n+1}(S/I).$

Therefore, it remains to show that $\beta_{n,n+2}(S/I)=0.$

Let $J=I:z.$ Let us consider the following exact sequence of graded $S$--modules:
\begin{equation}\label{sir1}
0\rightarrow \frac{S}{J}(-1) \stackrel{z}{\rightarrow} \frac{S}{I} \rightarrow \frac{S}{(I,z)} \rightarrow 0.
\end{equation}

We know that $S/I$ is Cohen-Macaulay of $\dim (S/I)=2$ by Proposition~\ref{CM}. On the other hand, 
$S/(I,z)\cong K[x,y_1,\ldots,y_n]/(y_iy_j: 1\leq i<j\leq n)$. The ring $K[y_1,\ldots,y_n]/(y_iy_j: 1\leq i<j\leq n)$ is the Stanley-Reisner ring of the simplicial complex whose facets are $\{y_1\},\ldots, \{y_n\}$ which is Cohen-Macaulay, thus $S/(I,z)$ is also Cohen-Macaulay of \[\dim(S/(I,z))=\dim  K[y_1,\ldots,y_n]/(y_iy_j: 1\leq i<j\leq n)+1=2.\] 

By applying Depth Lemma in (\ref{sir1}), it follows that 
$S/J$ is Cohen-Macaulay of $\dim (S/J)=2, $ thus $\projdim(S/I)=\projdim(S/J)=n.$

From the exact sequence (\ref{sir1}), we get the following exact $\Tor$--sequence:
\[
\Tor_n \Big(\frac{S}{J}(-1), K\Big)_{n+2}\rightarrow \Tor_n \Big(\frac{S}{I}, K\Big)_{n+2} \rightarrow \Tor_n \Big(\frac{S}{(I,z)}, K\Big)_{n+2}.
\]
Since $S/(I,z)\cong K[x,y_1,\ldots,y_n]/(y_iy_j: 1\leq i<j\leq n)$ and  the ideal $(y_iy_j: 1\leq i<j\leq n)$ has linear quotients, thus it has a linear resolution, 
it follows that  $\Tor_n \Big(\frac{S}{(I,z)}, K\Big)_{n+2}=0$. We have $\Tor_n \Big(\frac{S}{J}(-1), K\Big)_{n+2}=\Tor_n \Big(\frac{S}{J}, K\Big)_{n+1}$. Therefore, in order to get the desired vanishing of $\beta_{n,n+2}(S/I)$ it is enough to show that  
$\Tor_n \Big(\frac{S}{J}, K\Big)_{n+1}=0$. 

By using \cite[Lemma 12.1]{S}, a Gr\"obner basis for $J$ with respect to reverse lexicographic order induced by  $x>y_1>\ldots >y_n>z$, is given by \[\{y_iy_j-zx, x(y_k-y_n), x(y_n^2-xz): 1\leq i<j\leq n, 1\leq k\leq n-1\}.\] If we change the order of variables to $z>y_1>...>y_n>x$, the Gr\"obner basis for $J$ does not change. We can consider the following exact sequence:

\begin{equation}\label{sir2}
0\rightarrow \frac{S}{J:x}(-1) \stackrel{x}{\rightarrow} \frac{S}{J} \rightarrow \frac{S}{(J,x)} \rightarrow 0.
\end{equation}
From  (\ref{sir2}), we get the exact sequence
\[
\Tor_n \Big(\frac{S}{J:x}(-1), K\Big)_{n+1}\rightarrow \Tor_n \Big(\frac{S}{J}, K\Big)_{n+1} \rightarrow \Tor_n \Big(\frac{S}{(J,x)}, K\Big)_{n+1}.
\] Note that $\Tor_n \Big(\frac{S}{J:x}(-1), K\Big)_{n+1}=\Tor_n \Big(\frac{S}{J:x}, K\Big)_{n}=0.$ Indeed, by using again \cite[Lemma 12.1]{S}, we derive that $J:x$ is generated by the regular sequence $y_1-y_n,y_2-y_n,\ldots,y_{n-1}-y_n,y_n^2-xz,$ thus the minimal free  resolution of 
$S/J:x$ is provided by the Koszul complex. The last free module in the resolution is generated in degree $n+1,$ hence 
$\Tor_n \Big(\frac{S}{J:x}, K\Big)_{n}=0.$

 In addition, 
$S/(J,x)\cong K[y_1,\ldots,y_n,z](y_iy_j:1\leq i<j\leq n)$ is Cohen-Macaulay of dimension $2,$ thus $\projdim S/(J,x)=n-1$, so 
$\Tor_n \Big(\frac{S}{(J,x)}, K\Big)_{n+1}=0$. Therefore, $\Tor_n \Big(\frac{S}{J}, K\Big)_{n+1}=0$ which yields  $\beta_{n, n+2}(\frac{S}{I})=0$.
\end{proof}

\begin{Remark}\label{Rem2}{\em
We have seen in  the above theorem  that $S/I$ is a Gorenstein algebra. By Proposition~\ref{Hilbert},  the degree of the numerator of the Hilbert series is $3$ and the initial degree of $I$ is $2$, thus the algebra $S/I$ is a nearly extremal Gorenstein algebra in the sense of  \cite{Kumar}.
}
\end{Remark}

A homogeneous  ideal $I$ in a polynomial ring $R$ which is generated in degree $d$ is called \emph{linearly related} if its syzygy module 
is generated by linear relations. In other words, $I$ is linearly related if $\beta_{1j}(I)=\beta_{2j}(R/I)=0$ if $j\geq d+2.$  The next proposition shows that the join-meet ideal of the diamond lattice $D_{n+2}$ is not linearly related for every $n\geq 3.$

\begin{Proposition}\label{notlin}
Let $I\subset S=K[L]$ be the join-meet ideal of $D_{n+2}.$ Then $\beta_{24}(S/I) > 0.$
\end{Proposition}

\begin{proof}
We have seen in the proof of Theorem~\ref{Gorenstein} that the ideal $I:xz=(I:z):x=J:x$ is generated by the regular sequence 
$y_1-y_n,y_2-y_n,\ldots,y_{n-1}-y_n,y_n^2-xz$. Let us consider the exact sequence
\begin{equation}\label{eq4}
0\to \frac{S}{I:xz}(-2)\stackrel{xz}{\rightarrow}\frac{S}{I}\to \frac{S}{(I,xz)}\to 0
\end{equation}
and assume that $\beta_{24}(S/I)=\dim_K\Tor_2(S/I,K)_4=0.$ From the exact sequence (\ref{eq4}), we derive that
\[0\to \Tor_2\left(\frac{S}{(I,xz)},K\right)_4\to \Tor_1\left(\frac{S}{I:xz},K\right)_2\to \Tor_1\left(\frac{S}{I},K\right)_4=0
\]
is also an exact sequence. Therefore, we get
\[\Tor_2\left(\frac{S}{(I,xz)},K\right)_4\cong \Tor_1\left(\frac{S}{I:xz},K\right)_2,
\] which implies that {\small
\[\beta_{24}\left(\frac{S}{(I,xz)}\right)=\dim_K \Tor_2\left(\frac{S}{(I,xz)},K\right)_4=\dim_K \Tor_1\left(\frac{S}{I:xz},K\right)_2=\beta_{12}\left(\frac{S}{I:xz}\right)=1.
\]}
On the other hand, we have 
\[\frac{S}{(I,xz)}\cong \frac{K[y_1,\ldots,y_n]}{(y_iy_j: 1\leq i<j\leq n)}\otimes_K\frac{K[x,z]}{(xz)}\]
which implies that 
\[\beta_{24}\left(\frac{S}{(I,xz)}\right)=\beta_{12}\left(\frac{K[y_1,\ldots,y_n]}{(y_iy_j: 1\leq i<j\leq n)}\right)
=\binom{n}{2}.\]
Consequently, $\binom{n}{2}=1,$ thus, $n=2,$ a contradiction.
\end{proof}

\section{A lower bound for the regularity of the join-meet ideal of a modular non-distributive lattice}
\label{two}

 Let $L$ be a finite  lattice and $I_{L}\subset S=K[x_a : a\in L]$ its associated join-meet ideal. A sublattice $L^\prime$ of $L$ is called \textit{induced} if it has the following property: for every $a,b\in L,$ if $a\vee b, a\wedge b\in L^\prime,$ then $a,b\in L^\prime.$

Let us recall that, given the graded $K$--algebras $R^\prime\subset R,$ $R^\prime$ is an \textit{algebra retract} of $R$ if there exists a surjective homomorphism of graded algebras $\varepsilon: R\to R^\prime$ such that the restriction of $\varepsilon$ to $R^\prime$ is the identity of $R^\prime.$

\begin{Proposition}\label{retr}
Let $L$ be a finite lattice, $L^\prime$ an induced sublattice in $L$, and let  $I_{L^\prime}\subset K[x_a: a\in L^\prime]$ be its associated ideal. 
Then $R^\prime=S^\prime/I_{L^\prime}$ is an algebra retract of $R=S/I_L.$

\end{Proposition}

\begin{proof}
First, we show that $I_L\cap S^\prime=I_{L^\prime}.$ Obviously, $I_{L^\prime}\subset I_L\cap S^\prime. $ Let $f\in I_{L}\cap S^\prime$. Then 
$f=\sum_{a,b}(x_ax_b-x_{a\vee b}x_{a\wedge b})g_{ab}$, where the sum is taken over $a,b\in L,$ incomparable, and $g_{ab}\in S.$ 

We map
$x_a\mapsto 0$ for all $a\in L\setminus L^\prime. $ If $a\in L\setminus L^\prime,$ that is, $x_a\mapsto 0,$ then $a\vee b\notin L^\prime$
or $a\wedge b\notin L^\prime$ since $L^\prime$ is induced in $L.$ Hence, $x_{a\vee b}\mapsto 0$ or $x_{a\wedge b} \mapsto 0.$
Conversely, if $a\vee b\notin L^\prime$ or $a\wedge b\notin L^\prime,$ then $a\notin L^\prime$ or $b\notin L^\prime$ since 
$L^\prime$ is a sublattice in $L.$

Consequently, after mapping 
$x_a$ to $0$ for any $a\in L\setminus L^\prime,$ we get \[f=\sum_{b,c}(x_bx_c-x_{b\vee c}x_{b\wedge c})g^\prime_{bc}\] where the sum is taken over $b,c\in L^\prime$ incomparable, and $g^\prime_{bc}\in S^\prime,$ thus $f\in I_{L^\prime}.$

As $I_L\cap S^\prime=I_{L^\prime},$ we have an injective $K$--algebra homomorphism $R^\prime\to R$ and the map
$\varepsilon :R\to R^\prime$ induced by $x_a\mapsto 0$ for $a\in L\setminus L^\prime$ is a retraction map.

\end{proof}

\begin{Theorem}\label{mainthm}
Let $L$ be a finite modular non-distributive lattice $L$ and $I_L\subset S=K[L]$ its join-meet ideal.
Then 
\begin{itemize}
	\item [\emph{(i)}] $\reg{S/I_L}\geq 3;$
	\item [\emph{(ii)}] $\beta_{24}(S/I_L)>0$, thus $I_L$ is not a linearly related ideal.
\end{itemize}
 In particular, $S/I_L$ does not have a linear resolution.
\end{Theorem}

\begin{proof} Since $L$ is a modular non-distributive lattice, by  \cite[Lemma 1.2]{EnHi}, there is a diamond sublattice $L_1$ of $L$ such that 
$\rank \max L_1-\rank \min L_1=2.$ Let $x=\max L_1$ and $z=\min L_1.$ We consider the induced sublattice $L^\prime$ of $L$  with 
$x=\max L^\prime$ and $z=\min L^\prime.$ Then $L^\prime$ is a diamond sublattice of $L.$
 Let $I_{L^\prime}\subset S^\prime=K[L^\prime]$ its associated ideal. By Proposition~\ref{retr}, we know that $S^\prime/I_{L^\prime}$ is an algebra retract of $S/I_L$. By \cite[Corollary 2.5]{OHH}, 
we have $\beta_{ij}(S^\prime/I_{L^\prime})\leq \beta_{ij}(S/I_L)$ for all $i,j.$ In particular, 
$\reg(S^\prime/I_{L^\prime})\leq \reg(S/I_L).$ By Proposition~\ref{CM}, we have $\reg(S^\prime/I_{L^\prime})=3,$ thus the statement (i) holds. By using Proposition~\ref{notlin}, we get $\beta_{24}(S/I_L)\geq \beta_{24}(S^\prime/I_{L^\prime})>0,$ thus we have (ii).
\end{proof}

%It would be interesting to characterize the modular non-distributive lattices $L$ with the property that $\reg{S/I_L}\geq 3.$ Here, we 
%present some examples.

\begin{Example}
{\em Here there are some examples of regularity $3.$}

\begin{figure}[hbt]
\begin{center}
\psset{unit=0.35cm}
\begin{pspicture}(0.3,-5.5)(3,7)

\psline(-9,3)(-11,1)
\psline(-9,3)(-9,1)
\psline(-9,3)(-7,1)
\psline(-11,1)(-9,-1)
\psline(-9,1)(-9,-1)
\psline(-7,1)(-9,-1)
\psline(-9,3)(-11,5)
\psline(-13,3)(-11,5)
\psline(-13,3)(-11,1)
\psline(-11,1)(-13,-1)
\psline(-13,-1)(-11,-3)
\psline(-11,-3)(-9,-1)

\rput(-9,3){$\bullet$}
%\put(-9.3,3.3){$h$}
\rput(-11,1){$\bullet$}
%\put(-11.9,0.9){$d$}
\rput(-9,1){$\bullet$}
%\put(-8.6,0.9){$e$}
\rput(-7,1){$\bullet$}
%\put(-6.7,1){$f$}
\rput(-9,-1){$\bullet$}
%\put(-9.1,-1.6){$c$}
\rput(-11,-3){$\bullet$}
%\put(-11,-3.5){$a$}
\rput(-13,-1){$\bullet$}
%\put(-13.8,-1){$b$}
\rput(-13,3){$\bullet$}
%\put(-13.8,3){$g$}
\rput(-11,5){$\bullet$}
%\put(-11,5.5){$\ell$}
%\rput(-11,-4.5){ Lattice N}

\psline(2,3)(0,1)
\psline(2,3)(2,1)
\psline(2,3)(4,1)
\psline(0,1)(2,-1)
\psline(2,1)(2,-1)
\psline(4,1)(2,-1)
\psline(2,3)(0,5)
\psline(-2,3)(0,5)
\psline(-2,3)(0,1)
\psline(0,1)(-2,-1)
\psline(-2,-1)(0,-3)
\psline(0,-3)(2,-1)
\psline(-4,1)(-2,3)
\psline(-4,1)(-2,-1)

\rput(2,3){$\bullet$}
%\put(2.3,3.3){$h$}
\rput(0,1){$\bullet$}
%\put(0.9,0.9){$d$}
\rput(2,1){$\bullet$}
%\put(2.4,0.9){$e$}
\rput(4,1){$\bullet$}
%\put(4.3,1){$f$}
\rput(2,-1){$\bullet$}
%\put(1.9,-1.6){$c$}
\rput(0,-3){$\bullet$}
%\put(0,-3.5){$a$}
\rput(-2,-1){$\bullet$}
%\put(-2.8,-1){$b$}
\rput(-2,3){$\bullet$}
%\put(-2.8,3){$g$}
\rput(0,5){$\bullet$}
\rput(-4,1){$\bullet$}
%\put(-0,5.5){$\ell$}

\psline(13,3)(11,1)
\psline(13,3)(13,1)
\psline(13,3)(15,1)
\psline(11,1)(13,-1)
\psline(13,1)(13,-1)
\psline(15,1)(13,-1)
\psline(13,3)(11,5)
\psline(9,3)(11,5)
\psline(9,3)(11,1)
\psline(11,1)(9,-1)
\psline(9,-1)(11,-3)
\psline(11,-3)(13,-1)
\psline(11,5)(11,1)
%\psline(7,1)(9,3)
%\psline(7,1)(9,-1)

\rput(13,3){$\bullet$}
%\put(2.3,3.3){$h$}
\rput(11,1){$\bullet$}
%\put(0.9,0.9){$d$}
\rput(13,1){$\bullet$}
%\put(2.4,0.9){$e$}
\rput(15,1){$\bullet$}
%\put(4.3,1){$f$}
\rput(13,-1){$\bullet$}
%\put(1.9,-1.6){$c$}
\rput(11,-3){$\bullet$}
%\put(0,-3.5){$a$}
\rput(9,-1){$\bullet$}
%\put(-2.8,-1){$b$}
\rput(9,3){$\bullet$}
%\put(-2.8,3){$g$}
\rput(11,5){$\bullet$}
\rput(11,3){$\bullet$}
%\put(-0,5.5){$\ell$}
\end{pspicture}
\end{center}
\caption{Regularity $S/I_L$=3}\label{}
\end{figure}
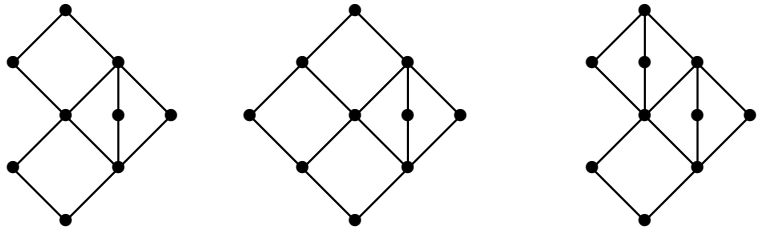
\end{Example}

\begin{Example}
{\em This example shows two  lattices of regularity $5.$}

\begin{figure}[hbt]
\begin{center}
\psset{unit=0.35cm}
\begin{pspicture}(-3.3,-5.5)(0,7)

\psline(-9,3)(-11,1)
%\psline(-9,3)(-9,1)
\psline(-9,3)(-7,1)
\psline(-11,1)(-9,-1)
%\psline(-9,1)(-9,-1)
\psline(-7,1)(-9,-1)
\psline(-9,3)(-11,5)
\psline(-13,3)(-11,5)
\psline(-13,3)(-11,1)
\psline(-11,1)(-13,-1)
\psline(-13,-1)(-11,-3)
\psline(-11,-3)(-9,-1)
\psline(-11,5)(-11,-3)

\rput(-9,3){$\bullet$}
\rput(-11,3){$\bullet$}
\rput(-11,-1){$\bullet$}
%\put(-9.3,3.3){$h$}
\rput(-11,1){$\bullet$}
%\put(-11.9,0.9){$d$}
%\rput(-9,1){$\bullet$}
%\put(-8.6,0.9){$e$}
\rput(-7,1){$\bullet$}
%\put(-6.7,1){$f$}
\rput(-9,-1){$\bullet$}
%\put(-9.1,-1.6){$c$}
\rput(-11,-3){$\bullet$}
%\put(-11,-3.5){$a$}
\rput(-13,-1){$\bullet$}
%\put(-13.8,-1){$b$}
\rput(-13,3){$\bullet$}
%\put(-13.8,3){$g$}
\rput(-11,5){$\bullet$}
%\put(-11,5.5){$\ell$}
%\rput(-11,-4.5){ Lattice N}

\psline(8,3)(6,1)
\psline(8,3)(8,1)
\psline(8,3)(10,1)
\psline(6,1)(8,-1)
\psline(8,1)(8,-1)
\psline(10,1)(8,-1)
\psline(8,3)(6,5)
\psline(4,3)(6,5)
\psline(4,3)(6,1)
\psline(6,1)(4,-1)
\psline(4,-1)(6,-3)
\psline(6,-3)(8,-1)
\psline(6,5)(6,1)
\psline(6,1)(6,-3)
%\psline(7,1)(9,3)
%\psline(7,1)(9,-1)

\rput(8,3){$\bullet$}
%\put(2.3,3.3){$h$}
\rput(6,1){$\bullet$}
%\put(0.9,0.9){$d$}
\rput(8,1){$\bullet$}
%\put(2.4,0.9){$e$}
\rput(10,1){$\bullet$}
%\put(4.3,1){$f$}
\rput(8,-1){$\bullet$}
%\put(1.9,-1.6){$c$}
\rput(6,-3){$\bullet$}
%\put(0,-3.5){$a$}
\rput(4,-1){$\bullet$}
%\put(-2.8,-1){$b$}
\rput(4,3){$\bullet$}
%\put(-2.8,3){$g$}
\rput(6,5){$\bullet$}
\rput(6,3){$\bullet$}
\rput(6,-1){$\bullet$}
%\put(-0,5.5){$\ell$}
\end{pspicture}
\end{center}
\caption{Regularity $S/I_L$=5}\label{}
\end{figure}
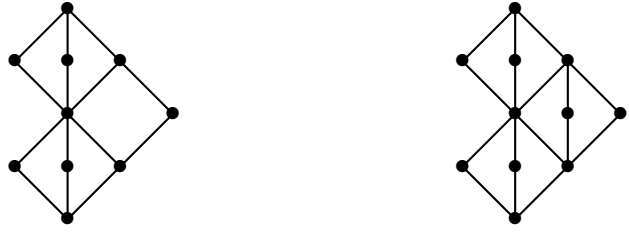

\end{Example}

\begin{Example}{\em 
The following class of non-distributive modular lattices was consi\-dered in \cite{EnHi}.

\begin{figure}[hbt]
\begin{center}
\psset{unit=0.5cm}
\begin{pspicture}(-1,-11)(5,2)

\psline(-9,1)(-11,-1)
\psline(-9,1)(-7,-1)
\psline(-9,-3)(-7,-1)
\psline(-11,-1)(-9,-3)
\psline(-5,-3)(-7,-5)
\psline(-7,-5)(-5,-7)
\psline(-5,-7)(-3,-5)
\psline(-3,-5)(-5,-3)
\psline(-3,-9)(-1,-7)
\psline(-1,-7)(1,-9)
\psline(1,-9)(-1,-11)
\psline(-1,-11)(-3,-9)
\psline[linestyle=dotted](-9,-3)(-7,-5)
\psline[linestyle=dotted](-7,-1)(-5,-3)
\psline[linestyle=dotted](-5,-7)(-3,-9)
\psline[linestyle=dotted](-3,-5)(-1,-7)

\rput(-9,1){$\bullet$}
\put(-8.5,1){$y_n$}
\rput(-11,-1){$\bullet$}
\put(-12.5,-1){$x_n$}
\rput(-7,-1){$\bullet$}
\put(-6.5,-1){$y_{n-1}$}
\rput(-9,-3){$\bullet$}
\put(-11,-3){$x_{n-1}$}
\rput(-5,-3){$\bullet$}
\put(-4.5,-3){$y_{k+1}$}
\rput(-7,-5){$\bullet$}
\put(-9,-5){$x_{k+1}$}
\rput(-3,-5){$\bullet$}
\put(-2.5,-5){$y_k$}
\rput(-5,-7){$\bullet$}
\put(-6.5,-7){$x_k$}
\rput(-1,-7){$\bullet$}
\put(-0.5,-7){$y_2$}
\rput(-3,-9){$\bullet$}
\put(-4.5,-9){$x_2$}
\rput(1,-9){$\bullet$}
\put(1.5,-9){$y_1$}
\rput(-1,-11){$\bullet$}
\put(-2.5,-11){$x_1$}

\rput(-7,-9){(a)}
\rput(-7,-10){Lattice $D$}

\psline(6,1)(4,-1)
\psline(6,1)(8,-1)
\psline(6,-3)(8,-1)
\psline(4,-1)(6,-3)
\psline(10,-3)(8,-5)
\psline(8,-5)(10,-7)
\psline(10,-7)(12,-5)
\psline(12,-5)(10,-3)
\psline(12,-9)(14,-7)
\psline(14,-7)(16,-9)
\psline(16,-9)(14,-11)
\psline(14,-11)(12,-9)
\psline[linestyle=dotted](6,-3)(8,-5)
\psline[linestyle=dotted](8,-1)(10,-3)
\psline[linestyle=dotted](10,-7)(12,-9)
\psline[linestyle=dotted](12,-5)(14,-7)
\psline(10,-3)(10,-5)
\psline(10,-5)(10,-7)

\rput(6,1){$\bullet$}
\put(6.5,1){$y_n$}
\rput(4,-1){$\bullet$}
\put(2.5,-1){$x_n$}
\rput(8,-1){$\bullet$}
\put(8.5,-1){$y_{n-1}$}
\rput(6,-3){$\bullet$}
\put(4,-3){$x_{n-1}$}
\rput(10,-3){$\bullet$}
\put(10.5,-3){$y_{k+1}$}
\rput(8,-5){$\bullet$}
\put(6,-5){$x_{k+1}$}
\rput(12,-5){$\bullet$}
\put(12.5,-5){$y_k$}
\rput(10,-7){$\bullet$}
\put(8.5,-7){$x_k$}
\rput(14,-7){$\bullet$}
\put(14.5,-7){$y_2$}
\rput(12,-9){$\bullet$}
\put(10.5,-9){$x_2$}
\rput(16,-9){$\bullet$}
\put(16.5,-9){$y_1$}
\rput(14,-11){$\bullet$}
\put(12.5,-11){$x_1$}
\rput(10,-5){$\bullet$}
\put(10.3,-5){$z$}

\rput(8,-9){(b)}
\rput(8,-10){Lattice $L_k$}

\end{pspicture}
\end{center}
\caption{}\label{chain}
\end{figure}

Let $P$ be the poset which is the sum of a chain and and extra point, and  $D=\MI(P)$ be the corresponding  distributive lattice with the elements labeled as in Figure~\ref{chain} (a). For every $1\leq k\leq n-1,$ we denote by $L_k$ the lattice of Figure~\ref{chain} (b). 
Let $I=I_{L_k}\subset S=K[x_1,\ldots,x_n,y_1,\ldots,y_n,z]$ be the join-meet ideal of $L_k.$ We consider the lexicographic order on $S$
induced by $x_1>\cdots >x_n> y_1>\cdots> y_n>z.$
As it was proved in \cite[Lemma 3.2]{EnHi}, the initial ideal of $I$ with respect to this order is generated by the following set of monomials:
\[
\{x_jy_i: 1\leq i< j\leq n\}\cup\{x_iy_{k+1}: 1\leq i< k\}\cup\{x_ky_j: k+1\leq j\leq n\}
\]
\[
\cup\{x_ix_{k+1}y_j,x_iy_ky_j: 1\leq i< k< k+1 < j\leq n\}\cup\{y_iy_kz: 1\leq i\leq k\}\cup\{x_{k+1}z\}.
\]

Let $J=\ini_<(I).$ We claim that $\reg(S/J)\leq 3.$ Then, $\reg(S/I)\leq 3.$ The opposite inequality follows by Theorem~\ref{mainthm}, thus we get $\reg(S/I)=\reg(S/\ini_<(I))=3.$ We now sketch the proof of inequality $\reg(S/J)\leq 3.$ We consider the following short exact sequences:
\begin{equation}\label{iniseq1}
0\to \frac{S}{J:y_{k+1}}(-1)\stackrel{y_{k+1}}{\rightarrow}\frac{S}{J}\to \frac{S}{(J,y_{k+1})}\to 0
\end{equation}
and 
\begin{equation}\label{iniseq2}
0\to \frac{S}{(J,y_{k+1}):y_k}(-1)\stackrel{y_{k}}{\rightarrow}\frac{S}{(J,y_{k+1})}\to \frac{S}{(J,y_{k+1},y_k)}\to 0.
\end{equation}

We have: 
\[
J:y_{k+1}=(x_1,\ldots,x_k,x_{k+2},\ldots,x_n)+x_{k+1}(y_1,\ldots,y_k,z)+(y_1y_kz,\ldots,y_{k-1}y_kz,y_k^2z),
\]
\[
(J,y_{k+1}):y_k=(y_{k+1},x_{k+1},x_{k+2},\ldots,x_n)+\sum_{j=1}^{k}x_j(y_{k+2},\ldots,y_n)+\sum_{j=2}^k x_{j}(y_1,\ldots,y_{j-1})\] \[ +
(y_1,\ldots,y_k)z,
\]
and
\[
(J,y_{k+1},y_k)=(y_{k+1},y_k) +(x_2y_1)+x_3(y_2,y_1)+\cdots + x_k(y_1,\ldots,y_{k-1})\]
\[
+x_{k+1}(y_1,\ldots,y_{k-1},z)+x_{k+2}(y_1,\ldots,y_{k-1})+x_{k+3}(y_1,\ldots,y_{k-1},y_{k+2})+\cdots\]  \[ + x_n(y_1,\ldots,y_{k-1},y_{k+2},\ldots,y_n)
\]

One may easily check that each of these  ideals have linear quotients, thus they are componentwise linear. Therefore, we get 
\[
\reg \frac{S}{J:y_{k+1}}=2, \reg \frac{S}{(J,y_{k+1}):y_k}=1, \text{ and }\reg\frac{S}{(J,y_{k+1},y_k)}=1.
\]
From the exact sequence (\ref{iniseq2}), we get $\reg S/(J,y_{k+1})=2,$ and, replacing in (\ref{iniseq1}), we derive that $\reg S/J\leq 3.$

}

\end{Example}

The above examples show that there are many lattices of minimal regularity. It would be interesting to find a characterization of the modular non-distributive lattices $L$ for which $\reg S/I_L=3.$

\end{document}